\documentclass{birkjour}
\def\RR{{\mathbb R}}
\def\ZZ{{\mathbb Z}}
\def\eb{{\bf e}}
\usepackage{amsmath, amssymb}
\usepackage{array}
\usepackage{float}

\newtheorem{thm}{Theorem}[section]
\newtheorem{prop}[thm]{Proposition}

\newtheorem{lemma}[thm]{Lemma}

\theoremstyle{definition}

\newtheorem{rei}[thm]{Example}

\title[Ehrhart series of fractional stable set polytopes]{Ehrhart series of fractional stable set polytopes of finite graphs}

\author{Ginji Hamano}
\address{%
Department of Pure and Applied Mathematics\\
Graduate School of Information Science and Technology\\
Osaka University\\
Suita, Osaka 565-0871\\
Japan}
\email{g-hamano@ist.osaka-u.ac.jp}

\author{Takayuki Hibi}
\address{%
Department of Pure and Applied Mathematics\\
Graduate School of Information Science and Technology\\
Osaka University\\
Suita, Osaka 565-0871\\
Japan}
\email{hibi@math.sci.osaka-u.ac.jp}

\author{Hidefumi Ohsugi}
\address{%
Department of Mathematical Sciences\\
School of Science and Technology\\
Kwansei Gakuin University\\
Sanda, Hyogo 669-1337\\
Japan} 
\email{ohsugi@kwansei.ac.jp}

\date{}

\begin{document}

\begin{abstract}
The fractional stable set polytope ${\rm FRAC}(G)$ of a simple graph $G$ with $d$ vertices is a rational polytope that is
the set of nonnegative vectors $(x_1,\ldots,x_d)$ satisfying $x_i+x_j\le 1$ for every edge $(i,j)$ of $G$.
In this paper we show that 
(i) The $\delta$-vector of 
a lattice polytope $2 {\rm FRAC}(G)$ is alternatingly increasing,
(ii) The Ehrhart ring of ${\rm FRAC}(G)$ is Gorenstein,
(iii) The coefficients of the numerator of the Ehrhart series of ${\rm FRAC}(G)$ are symmetric, unimodal and computed by
the $\delta$-vector of  $2 {\rm FRAC}(G)$.
\end{abstract}
\maketitle

\section*{Introduction} 

The Ehrhart series of a rational convex polytope is one of the most important topics in combinatorics.
Let $P$ be a $d$-dimensional rational convex polytope in $\mathbb{R}^N$.
For each $n \in \mathbb{N}$, let 
$nP = \{n \alpha  \mid  \alpha \in P\}$ and define the function
$i(P,n):=\sharp ( nP\cap \mathbb{Z}^N ).$
Thus $i(P,n)$ is the number of lattice points contained in $nP$, called the 
{\em Ehrhart quasi-polynomial} of $P$.
It is known that $i(P,n)$ is indeed a quasi-polynomial of degree $d$.
In particular, if $P$ is a lattice polytope, i.e., all vertices of $P$ are lattice points,
then $i(P,n)$ is a polynomial and called the {\em Ehrhart polynomial} of $P$.
The generating function of the Ehrhart quasi-polynomial is defined by
$E(P,t):=1+\sum_{n=1}^{\infty} i(P,n)t^n$
and called the {\em Ehrhart series} of $P$.
Let $m$ be the smallest natural number $k$ for which $k P$ is a lattice polytope
and let $v$ be the smallest natural number $k$ for which $kP$ has a lattice point in its interior.
It is known that $E(P,t)$ is a rational function of degree $-v$ and has an expression
$E(P,t):=g(P,t)/(1-t^m)^{d+1}$
where $g(P,t)$ is a polynomial of degree $m(d+1) -v$ with nonnegative integer coefficients.
In particular, if $P$ is a lattice polytope, then $m=1$ and hence
$E(P,t)=\delta(P,t)/(1-t)^{d+1}$
where $\delta(P,t) = \delta_0 + \delta_1 t + \cdots + \delta_{d+1-v}
t^{d+1-v}$ is a polynomial of degree $d+1 -v$, called the {\em $\delta$-polynomial} of $P$.
The coefficients $(\delta_0,\ldots,\delta_{d+1-v})$ of $\delta(P,t)$
is called the {\em $\delta$-vector} (or $h$-vector, $h^*$-vector) of $P$.
Next, we define the Ehrhart ring of a $d$-dimensional rational polytope
$P\subset \mathbb{R}^N$.
Let $[A_K(P)]_n$ be the linear space over a field $K$ whose basis is 
the set of Laurent monomials $ x_1^{\alpha_1} \cdots x_N^{\alpha_N} t^n$
with $(\alpha_1,\ldots, \alpha_N) \in nP
 \cap \mathbb{Z}^N$.
Then
$A_K(P):=\bigoplus_{n\ge 0}[A_K(P)]_n$
is called the {\em Ehrhart ring} of $P$.
If $A_K(P)$ is Gorenstein, then
the coefficients of the numerator of $E(P,t)$ are symmetric.
The {\em dual polytope} of $P$ is defined by 
$P^{\vee}=\{{\bf x} \in\mathbb{R}^d \ | \ \langle {\bf x}, {\bf y} \rangle\le 1
\mbox{ for all } {\bf y}\in P\}$, where $\langle {\bf x}, {\bf y} \rangle$ is the usual inner product of $\mathbb{R}^d$. 
The notion of dual polytopes appears in a criterion for $A_K(P)$ to be Gorenstein.
Let $P\subset \mathbb{R}^d$ be a lattice polytope of dimension $d$.
We say that $P$ is a {\em Fano polytope} if the origin of $\mathbb{R}^d$ is the unique lattice point belonging to the interior of $P$.
A Fano polytope is called {\em Gorenstein} if its dual polytope is a lattice polytope.
(A Gorenstein Fano polytope is often called a {\em reflexive polytope} in the literature.)

Let $G$ be a finite simple graph on the vertex set $[d] = \{1,2,\ldots,d\}$
and let $E(G)$ be the edge set of $G$.
Throughout this paper, we always assume that $G$ has no isolated vertices.
Given a subset $W \subset [d]$, we associate the $(0,1)$-vector
$\rho(W) = \sum_{j \in W} {\eb}_j \in {\RR}^d.$
Here, ${\bf e}_i$ is the $i$th unit coordinate vector of ${\RR}^d$. 
In particular, $\rho(\emptyset)$ is the origin of $\RR^d$.
A subset $W$ is called {\em stable} if $\{i,j\} \notin E(G)$ for all $i,j \in W$
with $i \neq j$.
Note that the empty set and each single-element subset of $[d]$ are stable.
Let $S(G)$ denote the set of all stable sets of $G$.
The {\em stable set polytope} ({\em independent set polytope}) ${\rm STAB} (G) \subset \RR^d$ 
of a simple graph $G$ is the $(0,1)$-polytope
which is the convex full of $\{\rho(W) \ | \  W \in S(G)\}$.
Stable set polytopes are very important in many areas, e.g., 
optimization theory.
The $\delta$-vector of the stable set polytope of 
a perfect graph is studied in
\cite{A, OH}.
On the other hand, the {\em fractional stable set polytope} ${\rm FRAC} (G)$ of $G$ is the $d$-polytope in $\mathbb{R}^d$ defined by
$$ 
{\rm FRAC} (G):=
\left\{(x_1,\ldots,x_d) \in \RR^d
\left|
\begin{array}{cc}
x_i \ge 0 & (1\le i \le d)\\
x_i+x_j\le 1 & ( (i,j)\in E(G) )
\end{array}
\right.
\right\}.
$$
In general, we have
${\rm STAB}(G) \subset {\rm FRAC} (G)$.
Each vertex of ${\rm FRAC}(G)$ belongs to $\{0,1/2,1\}^d$
(see, e.g, \cite{NT}).
It is known that ${\rm FRAC} (G) = {\rm STAB}(G)$ if and only if $G$ is bipartite.
If $G$ is bipartite, then ${\rm STAB}(G)$ has a unimodular
triangulation, and
the $\delta$-vector of ${\rm STAB}(G)$ 
is symmetric and unimodal (see \cite{A, BR, OH}).
Note that, if $G$ is bipartite, then ${\rm STAB}(G)$ 
 is the {\em chain polytope} of a poset $P$ of rank $1$
whose comparability graph is $G$,
and affinely equivalent to the {\em order polytope}
of the poset $P$ (see \cite{Rich}).
The purpose of this paper is to study the Ehrhart series of ${\rm FRAC}(G)$.
The following two polytopes will play important roles:
\begin{eqnarray*}
{\mathcal P}(G) &=& 2 \cdot {\rm FRAC}(G),\\
{\mathcal Q}(G) &=& 3 \cdot {\rm FRAC}(G) - (1,\ldots,1)\\
&=&
\left\{
(x_1,\ldots,x_d) \in \RR^d
\left|
\begin{array}{cc}
 x_i \ge -1 & (1\le i \le d)\\
 x_i+x_j\le 1 & ((i,j)\in E(G))
\end{array}
\right.
\right\}.
\end{eqnarray*}
In \cite{Steingrimsson}, Steingr\'{i}msson called
the lattice polytope ${\mathcal P}(G)$ 
the {\em extended $2$-weak vertex-packing polytope} of $G$
and studied the structure of ${\mathcal P}(G)$.
In particular, he constructed a unimodular triangulation of ${\mathcal P}(G)$ and
showed that the $\delta$-vector of ${\mathcal P}(G)$
is obtained by a descent statistic on a subset of the hyperoctahedral group determined by $G$.

This paper is organized as follows.
In Section 1, we show that the $\delta$-vector $(\delta_0,\ldots, \delta_{d-1})$ of ${\mathcal P}(G)$
is {\em alternatingly increasing} (\cite[Definition 2.9]{SV}), i.e.,
$$
\delta_0 \le \delta_{d-1} \le \delta_1 \le \delta_{d-2}
\le \cdots \le
\delta_{\lfloor d/2 \rfloor-1} \le
\delta_{d- \lfloor d/2 \rfloor} \le
\delta_{\lfloor d/2 \rfloor}.
$$
In Section 2, we study the structure of ${\mathcal Q}(G)$ in order to show
that the Ehrhart ring of ${\rm FRAC} (G)$ is Gorenstein.
By using this result, in Section 3, we give a formula for
the numerator of the Ehrhart series
$E({\rm FRAC} (G),t):=g({\rm FRAC} (G),t)/(1-t^2)^{d+1}$
via the $\delta$-vector of ${\mathcal P}(G)$.
Since the Ehrhart ring of ${\rm FRAC} (G)$ is Gorenstein
and since the $\delta$-vector of ${\mathcal P}(G)$
is alternatingly increasing, it follows that the coefficients of $g({\rm FRAC} (G),t)$
is symmetric and unimodal.
Finally, in Section 4, we discuss the dual polytope ${\mathcal Q}(G)^{\vee}$
of ${\mathcal Q}(G)$.

\section{The $\delta$-vector of ${\mathcal P}(G)$}

First, we review the results in \cite{Steingrimsson}.
Let $B_d$ denote the all signed permutation words on $[d]=\{1,2,\ldots,d\}$.
For example, if $d=2$,
$$
B_2=
\{
1\ 2,2\ 1,\overline{1}\ 2,2 \ \overline{1},1\ \overline{2},\overline{2}\ 1, \overline{1}\ \overline{2},\overline{2} \ \overline{1}
%1\ 2,2\ 1, {-1}\ 2,2 \ {-1},1\ {-2}, {-2}\ 1, {-1}\ {-2}, {-2} \ {-1}
\},
$$
where $\overline{1}=-1$ and $\overline{2}=-2$.
We order the letters in signed permutations as integers,
i.e., $ \cdots<  \overline{3} < \overline{2}< \overline{1}
<0 < 1 < 2 < 3 < \cdots $.
An element $i \in [d]$ is called a {\em descent} 
in $\pi = a_1 \cdots a_d \in B_d$ if one of the following holds
(\cite[Definition 5]{Steingrimsson}):
\begin{itemize}
\item[(i)]
$i < d$ and $a_i > a_{i+1}$;
\item[(ii)]
$i = d$ and $a_i > 0$.
\end{itemize}
Let ${\rm des} (\pi)$ denote the number of descents in
$\pi \in B_d$.
For example, for $\pi = 2  \ \overline{3}\   \overline{4}\ 1\in B_4$, 
 ${\rm des} (\pi) = 3$ since the descents of $\pi$ are $1$, $2$ and $4$.
For any subset $S$ of $B_d$,
the {\em descent polynomial} of $S$ is 
$D(S, t):= \sum_{\pi \in S} t^{{\rm des}(\pi)}$.
Let $G$ be a simple graph on the vertex set $[d]$ and the edge set $E(G)$.
We define a subset $\Pi(G)$ of $B_d$ as follows (\cite[Definition 11 and Theorem 12]{Steingrimsson}):
$$
\Pi(G)=
\left\{
\pi \in B_d\ 
\left| \ 
\begin{array}{c}
\mbox{if } (i, j) \in E(G) \mbox{ and } +i 
\mbox{ appears in } \pi, \\
\mbox{ then } -j 
\mbox{ must precede } +i \mbox{ in } \pi
\end{array}
\right.
\right\}.
$$

\begin{prop}[\cite{Steingrimsson}]
Let $G$ be a finite simple graph.
Then the $\delta$-polynomial of ${\mathcal P}(G)$ equals the descent polynomial
$D(\Pi(G),t)$.
\end{prop}

By using this fact, we will show Theorem \ref{ai} below.
Note that a similar decomposition technique (i.e., $a(t) + tb(t)$) was used in \cite{Sta}
to establish Ehrhart inequalities originally due to Stanley and Hibi.

\begin{thm}
\label{ai}
Let $G$ be a simple graph with $d$ vertices.
Then there exist symmetric and unimodal polynomials
$a(t)$ of degree $d-1$ and $b(t)$ of degree $d-2$
such that $\delta({\mathcal P}(G),t) =  a(t) + tb(t)$.
In particular, the $\delta$-vector $(\delta_0,\delta_1,\ldots,\delta_{d-1})$
of ${\mathcal P}(G)$ is alternatingly increasing, i.e.,
$$
\delta_0 \le \delta_{d-1} \le \delta_1 \le \delta_{d-2}
\le \cdots \le
\delta_{\lfloor d/2 \rfloor-1} \le
\delta_{d- \lfloor d/2 \rfloor} \le
\delta_{\lfloor d/2 \rfloor}.
$$
\end{thm}

\begin{proof}
Let $\Pi_+$ (resp. $\Pi_-$) denote the set of all 
$\pi \in \Pi(G)$ such that the last number of $\pi$
is positive (resp. negative).
Note that the first number of $\pi \in \Pi(G)$ is always negative
since $G$ has no isolated vertices.

Let $\pi \in \Pi_+$.
Then $\pi$ has a representation
$$
\pi = 
m_1^{(1)} \! \cdots m_{\alpha_1}^{(1)} 
p_1^{(1)} \! \cdots p_{\beta_1}^{(1)} 
m_1^{(2)}\! \cdots m_{\alpha_2}^{(2)} 
p_1^{(2)}\! \cdots p_{\beta_2}^{(2)} 
\cdots
m_1^{(\gamma)} \! \cdots m_{\alpha_\gamma}^{(\gamma)} 
p_1^{(\gamma)}\! \cdots p_{\beta_\gamma}^{(\gamma)} 
$$
where $p_i^{(j)} > 0$ and $m_i^{(j)}<0$.
Let $S(\pi)$ denote the set of all signed permutation words on $[d]$ of the form
$$
m_{\sigma_1(1)}^{(1)} \cdots m_{\sigma_1(\alpha_1)}^{(1)} 
p_{\tau_1(1)}^{(1)} \cdots p_{\tau_1(\beta_1)}^{(1)} 
\cdots
m_{\sigma_\gamma(1)}^{(\gamma)} \cdots m_{\sigma_\gamma(\alpha_\gamma)}^{(\gamma)} 
p_{\tau_\gamma(1)}^{(\gamma)} \cdots p_{\tau_\gamma(\beta_\gamma)}^{(\gamma)} 
$$
where 
$\sigma_k \in {\mathcal S}_{\alpha_k}$
and 
$\tau_k \in {\mathcal S}_{\beta_k}$
are permutations.
It is easy to see that $S(\pi) \subset \Pi_+$.
Let
$ A_k(t)=\sum_{i=0}^{k-1} A(k,i)t^i$
denote the {\em Eulerian polynomial} whose coefficients $A(k,i)$ is {\em Eulerian number}.
It is known that
$A_k(t) = \sum_{\pi \in {\mathcal S}_k} t^{d(\pi)} $,
where $d(\pi)$ is the number of {\em usual} descent of $\pi$
(i.e., no signs involved and never a descent at $k$).
See, e.g., \cite{Rich2}.
Thus
$$
D(S(\pi),t) = 
t^{\gamma}
\prod_{j=1}^\gamma 
A_{\alpha_j}(t) A_{\beta_j}(t).
$$
%where
%$ A_k(t)=\sum_{i=0}^{k-1} A(k,i)t^i$
%is the Eulerian polynomial whose coefficients $A(k,i)$ is {\em Eulerian number}.
It is known that 
%the sequence of Eulerian number 
$(A(k,0),A(k,1),\dots,A(k,k-1))$ is symmetric and
{\em unimodal}, i.e., 
$A(k,i)\le A(k,i+1)$
for $0\le i \le \lfloor 2/k \rfloor$.
The degree of $D(S(\pi),t)$
is $\gamma +\sum_{j=1}^\gamma (\alpha_j + \beta_j -2) = d -\gamma$.
Since $A_k(t) $ is symmetric and unimodal,
so is $A_{\alpha_j}(t) A_{\beta_j}(t) $.
Hence
$$
D(S(\pi),t) = s_\gamma t^\gamma + \cdots + s_{d -\gamma} t^{d -\gamma}
$$
satisfies that
$(s_\gamma,\ldots,s_{d-\gamma})$
is symmetric and unimodal.
Since
$$D(\Pi_+,t) = u_1 t  + \cdots + u_{d -1} t^{d -1}$$
is a sum of such $D(S(\pi),t)$'s,
$(u_1,\ldots,u_{d-1})$ is symmetric and unimodal.

Let $\pi \in \Pi_-$.
Then $\pi$ has a representation
$$\pi = 
m_1^{(1)} \cdots m_{\alpha_1}^{(1)} 
p_1^{(1)} \cdots p_{\beta_1}^{(1)} 
m_1^{(2)} \cdots m_{\alpha_2}^{(2)} 
p_1^{(2)} \cdots p_{\beta_2}^{(2)} 
\cdots
m_1^{(\gamma)} \cdots m_{\alpha_\gamma}^{(\gamma)},
$$
where $p_i^{(j)} > 0$ and $m_i^{(j)}<0$.
Define $S(\pi)$ as before.
Then we have
$$
D(S(\pi),t) = 
t^{\gamma-1}
A_{\alpha_\gamma}(t) 
\prod_{j=1}^{\gamma-1} 
A_{\alpha_j}(t) A_{\beta_j}(t).
$$
The degree of $D(S(\pi),t)$ is 
$\gamma -1+\alpha_\gamma-1 +\sum_{j=1}^{\gamma-1} (\alpha_j + \beta_j -2) = d -\gamma$.
Since
$$D(\Pi_-,t) = v_0+v_1 t  + \cdots + v_{d -1} t^{d -1}$$
is a sum of such $D(S(\pi),t)$'s,
$(v_0,v_1,\ldots,v_{d-1})$ is symmetric and unimodal.

We now show that the $\delta$-vector
$(\delta_0,\ldots,\delta_{d-1}) = (v_0,u_1+v_1,\ldots,u_{d-1}+v_{d-1})$
of ${\mathcal P}(G)$ is alternatingly increasing.
First, $\delta_{d-1} - \delta_0 = u_{d-1} + v_{d -1} -v_0 =u_{d-1} \ge 0 $.
Moreover, for $i = 1,2,\ldots, \lfloor d/2 \rfloor$,
we have $\delta_i - \delta_{d-i} =  u_i + v_i - u_{d-i} - v_{d -i}
=v_i  - v_{i-1} \ge 0$, and for 
$i = 1,2,\ldots, \lfloor d/2 \rfloor-1$,
we have
$\delta_{d-i-1} - \delta_i= u_{d-i-1} + v_{d -i-1} - u_i - v_i
=u_{i+1}  - u_i \ge 0$.
Thus the $\delta$-vector of ${\mathcal P}(G)$ is alternatingly increasing.
\end{proof}

\section{The Ehrhart ring of ${\rm FRAC}(G)$} 

In this section, we will show that the Ehrhart ring of ${\rm FRAC}(G)$ is Gorenstein.
In order to show that the Ehrhart ring of ${\rm FRAC}(G)$ is Gorenstein,
we will use the following criterion 
\cite[Theorem 1.1]{Hibi-Em}:

\begin{prop}
\label{criterionG}
Let $P \subset \mathbb{R}^d$ be a rational convex polytope of dimension $d$ and let $\delta \ge 1$ denote the smallest integer for which $\delta (P \setminus \partial P)\cap\mathbb{Z}^d \neq \emptyset$. Fix $\alpha \in \delta(P \setminus \partial P)\cap\mathbb{Z}^d$ and 
let $Q = \delta P-\alpha\subset\mathbb{R}^d$.
% of standard type. 
Then the Ehrhart ring $A_K(P)$ of $P$ is Gorenstein if and only if the following conditions are satisfied{\rm :}
\begin{enumerate}
\item [{\rm (i)}] The dual polytope $Q^{\vee}$ of $Q$ is a lattice polytope{\rm ;}
\item [{\rm (ii)}] Let $\widetilde{P} \subset \mathbb{R}^{d+1}$ denote the rational convex polytope which is the convex hull of the subset $\{(\beta,0)\in \mathbb{R}^{d+1} \mid \beta \in P\}\cup\{(0,\dots,0,1/\delta)\}$ in $\mathbb{R}^{d+1}$.
Then $\widetilde{P}$ is facet-reticular, that is to say, if $H$ is a hyperplane in $\mathbb{R}^{d+1}$ and if $H \cap \widetilde{P}$ is a facet of $\widetilde{P}$, then $H \cap \mathbb{Z}^{d+1} \neq \emptyset $.
\end{enumerate}
\end{prop}
It is clear that there exists no lattice points in the interior of 
%${\rm FRAC}(G)$ and 
${\mathcal P}(G)=2{\rm FRAC}(G)$, and that the lattice point $(1,\dots,1)$ belongs to the interior of $3{\rm FRAC}(G)$.
Thus it is enough to show that conditions (i) and (ii) in Proposition \ref{criterionG} 
are satisfied when
$P = {\rm FRAC}(G)$, $\delta=3$, $\alpha=(1,\dots,1)$ and $Q ={\mathcal Q}(G)$.
A criterion for a vector to be a vertex of ${\rm FRAC}(G)$
is given in \cite[Theorem 15]{Steingrimsson}:

\begin{lemma}
\label{usefullemma}
Let $G$ be a finite simple graph with $d$ vertices.
Suppose that ${\bf v} =(v_1,\dots,v_d) \in \{0,1/2,1\}^d$ belongs to
 ${\rm FRAC}(G)$.
Let $G_S$ be the subgraph of $G$ induced by $S=\{i \in [d] \mid v_i=1/2\}$.
Then ${\bf v}$ is a vertex of ${\rm FRAC}(G)$ if and only if
either $S=\emptyset$ or each connected component of $G_S$ contains an odd cycle.
\end{lemma}

Using Lemma \ref{usefullemma}, we determine when ${\mathcal Q}(G)$ is a lattice polytope.

\begin{prop}
\label{whenislattice}
Let $G$ be a finite simple graph without isolated vertices. 
Then the following conditions are equivalent.
\begin{enumerate}
\item [{\rm (i)}] The graph $G$ is a bipartite graph{\em ;}
\item [{\rm (ii)}] The polytope ${\rm FRAC}(G)$ is a lattice polytope{\em ;}
\item [{\rm (iii)}] The polytope ${\mathcal Q}(G)$ is a lattice polytope.
\end{enumerate}
\end{prop}

\begin{proof}
If $G$ is bipartite, then ${\rm FRAC}(G) = {\rm STAB}(G)$ is a lattice polytope.
Hence {\rm (i) $\Rightarrow$ (ii)} holds.
Moreover, {\rm (ii) $\Rightarrow$ (iii)} is trivial.
We now show that
{\rm (iii) $\Rightarrow$ (i)}.
Suppose $G$ contains an odd cycle $C$.
Let $H$ be a connected component of $G$ that contains $C$ and let $V(H)$ be the set of vertices of $H$.
Here, we define ${\bf v}=(v_1,\dots,v_d)$ by $v_i=1/2$ if $i\in V(H)$ and $v_i=0$ if $i \notin V(H)$. 
Then ${\bf v}$ is a $(0,1/2)$-vector in ${\rm FRAC}(G)$.
Moreover, since ${\bf v}$ satisfies the condition in Lemma \ref{usefullemma},
${\bf v}$ is a vertex of ${\rm FRAC}(G)$.
Then $3 {\bf v} - (1,\ldots,1) \in \{-1,1/2\}^d$ is a vertex of ${\mathcal Q}(G)$ that is not a lattice point.
Hence ${\mathcal Q}(G)$ is not a lattice polytope.
\end{proof}

Next we show that ${\mathcal Q}(G)^\vee$ is a lattice polytope.

\begin{prop}
\label{GF}
Suppose $G$ is a finite simple graph without isolated vertices.
Then the origin of $\mathbb{R}^d$ is a unique lattice point belonging to the interior of ${\mathcal Q}(G)$ and
$$\{{\bf e}_i+ {\bf e}_j \mid (i,j) \in E(G)\} \cup \{-{\bf e}_i \mid 1\le i \le d\}$$
is the vertex set of ${\mathcal Q}(G)^{\vee}$.
In particular, if $G$ is a bipartite graph, then ${\mathcal Q}(G)$ is a Gorenstein Fano polytope.
\end{prop}

\begin{proof}
It is known that
the inequalities $x_i \ge 0$ $(1 \le i \le d)$ and $x_i + x_j \le 1$ ($(i,j) \in E(G)$)
define the facets of ${\rm FRAC}(G)$.
Hence the inequalities $x_i \ge -1$ $(1 \le i \le d)$ and $x_i + x_j \le 1$ ($(i,j) \in E(G)$)
define the facets of ${\mathcal Q}(G)$.
Thus a vector $(v_1,\dots,v_d) \in \RR^d$ belongs to the interior of ${\mathcal Q}(G)$ 
if and only if $v_i > -1$ $(1 \le i \le d)$ and $v_i + v_j < 1$ ($(i,j) \in E(G)$).
It is clear that the origin of $\mathbb{R}^d$ belongs to the interior of ${\mathcal Q}(G)$.
Suppose that $(v_1,\dots,v_d)\in \ZZ^d$ belongs to the interior of ${\mathcal Q}(G)$.
Since $v_i$ and $v_i + v_j$ are integers, we have  $v_i \ge 0$ $(1 \le i \le d)$ and $v_i + v_j \le 0$ ($(i,j) \in E(G)$).
Hence $v_i=0$ for all $i$, i.e, $(v_1,\dots,v_d) = {\bf 0}$.
%the lattice belonging to the interior of ${\mathcal Q}(G)$ is only the origin of $\mathbb{R}^d$.
%
It is known that there is a one-to-one correspondence
between 
the facets of ${\mathcal Q}(G)$
and
the vertices of ${\mathcal Q}(G)^{\vee}$.
The set 
$\{{\bf e}_i+ {\bf e}_j \mid (i,j) \in E(G)\} \cup \{-{\bf e}_i \mid 1\le i \le d\}$
of coefficient vectors of inequalities that define facets 
 is the set of vertices of ${\mathcal Q}(G)^{\vee}$.
Thus, in particular, ${\mathcal Q}(G)^{\vee}$ is a lattice polytope.
By Proposition~\ref{whenislattice},
if $G$ is a bipartite graph, then ${\mathcal Q}(G)$ is a lattice polytope,
and hence a Gorenstein Fano polytope.
\end{proof}

We are now in the position to show that 
the Ehrhart ring of ${\rm FRAC}(G)$ is Gorenstein.

\begin{thm}
Let $G$ be a finite simple graph without isolated vertices. 
Then the Ehrhart ring of ${\rm FRAC}(G)$ is Gorenstein.
\end{thm}

\begin{proof}
It is enough to show that conditions (i) and (ii) in Proposition \ref{criterionG} 
are satisfied when
$P = {\rm FRAC}(G)$, $\delta=3, \alpha=(1,\dots,1)$ and $Q ={\mathcal Q}(G)$.
First, Proposition \ref{GF} guarantees that $Q^\vee$ is a lattice polytope.
Let $F=H \cap \widetilde{P}$ be a facet of $\widetilde{P}$,
where $H$ is a hyperplane in $\mathbb{R}^{d+1}$.
We may assume that $H \neq \{x_{d+1}=0\}$.
Then $F'=F\cap\{x_{d+1}=0\}$ is a facet of $\{(\beta,0)\in \mathbb{R}^{d+1} \mid \beta \in P\}$ whose supporting hyperplane is $H$.
Therefore, $H'=H\cap\{x_{d+1}=0\}$ is defined by $\{x_{d+1}=0\}$ and
either $x_i+x_j =1$ $((i,j)\in E(G))$ 
or $x_i=0$ $(1\le i \le d)$.
Hence it is clear that there exists a lattice point in $H' \subset H$.
Thus %there exists a lattice point in $H$ and
condition (ii) in Proposition~\ref{criterionG} holds.
Therefore, the Ehrhart ring of $P$ is Gorenstein by Proposition~\ref{criterionG}.
\end{proof}

\section{The Ehrhart series of ${\rm FRAC}(G)$}

In this section, we show that we can calculate the Ehrhart series and the Ehrhart quasi-polynomial
of ${\rm FRAC}(G)$ from that of ${\mathcal P}(G)$.
Let $G$ be a simple graph on the vertex set $[d]$ without isolated vertices.
Since the interior of ${\mathcal P}(G)$ possesses no lattice points, and the interior of $2 {\mathcal P}(G)$ has a lattice point,
it follows that $\deg \delta({\mathcal P}(G),t) = d+1 -2= d-1$.
On the other hand, the degree of $E({\rm FRAC}(G),t)$ is $-3$ as a rational function.
Given a rational convex polytope ${\mathcal P}$,
the period of $i({\mathcal P}, n)$ is a divisor of the smallest positive integer $\alpha$
for which $\alpha {\mathcal P}$ is a lattice polytope.
See \cite[Theorem 4.6.25]{Rich}.
Hence $i({\rm FRAC}(G), n)$ is a quasi-polynomial of period at most 2.
Thus there exist polynomials $i^{\rm odd}({\rm FRAC}(G),n)$ and $i^{\rm even}({\rm FRAC}(G),n)$ of degree $d$ such that
$$
i({\rm FRAC}(G),n)
=\begin{cases} 
i^{\rm odd}({\rm FRAC}(G),n) & \mbox{if } n \mbox{ is odd,} \\
\\
i^{\rm even}({\rm FRAC}(G),n) & \mbox{if } n \mbox{ is even.} 
\end{cases}
$$ 

\medskip

\noindent
In particular, if $G$ is bipartite, then $i^{\rm odd}({\rm FRAC}(G),n) = i^{\rm even}({\rm FRAC}(G),n)$.

\begin{thm}
\label{hpoly}
Let $G$ be a simple graph on the vertex set $[d]$ without isolated vertices
and let $\delta({\mathcal P}(G), t)=\delta_0+\delta_1t+\dots+\delta_{d-1}t^{d-1}$.
Then we have
\begin{align*}
E({\rm FRAC}(G),t) & =  
\frac{\delta({\mathcal P}(G), t^2) + t^{2d-1} \delta({\mathcal P}(G), 1/t^2)  }{(1-t^2)^{d+1}}\\
       & = 
\frac{\delta_0 + \delta_{d-1}t + \delta_1t^2 + \delta_{d-2}t^3 + \dots 
%+ \delta_1 t^{2d-3} 
+ \delta_{d-1}t^{2d-2} + \delta_0t^{2d-1}}{(1-t^2)^{d+1}},
\end{align*}
where $
(\delta_0,\delta_{d-1}, \delta_1,\delta_{d-2},\ldots,\delta_{d-1},\delta_0)
$ is symmetric and unimodal.
In addition, 
\begin{eqnarray*}
i^{\rm odd}({\rm FRAC}(G),2k+1)
&=& (-1)^{d} i^{\rm even}({\rm FRAC}(G),-2k-4)\\
&=& (-1)^{d} i({\mathcal P}(G),-k-2).
\end{eqnarray*}
\end{thm}

\begin{proof}
Let $W={\rm FRAC}(G)$ and $P= {\mathcal P}(G)$.
%, and let $E(W,t):=\sum_{n \ge 0} i(W,n)t^n$ be the Ehrhart series of $W$.
Then
$$E(W,t)=\sum_{k \ge 0} i^{\rm even}(W,2k)t^{2k}+\sum_{k \ge 0} i^{\rm odd}(W,2k+1)t^{2k+1}.$$
Since $i^{\rm even}(W,2k) =i(2W,k)  =i(P,k)$, we have 
%$\sum_{k \ge 0} i(W,2k)t^{2k}$ by
$$\sum_{k \ge 0} i^{\rm even}(W,2k)t^{2k}
=\sum_{k \ge 0} i(P,k)(t^2)^k=\dfrac{\delta(P,t^2)}{(1-t^2)^{d+1}}.$$
Since the degree of $i^{\rm odd}(W,2k+1)$ is $d$,
by \cite[Corollary 4.3.1]{Rich2}, we have
$$\sum_{k\ge 0} i^{\rm odd}(W,2k+1)t^{2k+1}
%=\sum_{k\ge0}f(2k+1)t^{2k+1}
=t \sum_{k\ge0}  i^{\rm odd}(W,2k+1) (t^2)^k
=t\dfrac{a(t^2)}{(1-t^2)^{d+1}}$$
where $a(t)$ is a polynomial of degree $\le d$.
Thus
$$E(W,t)=\dfrac{\delta(P,t^2)}{(1-t^2)^{d+1}}+t\cdot\dfrac{a(t^2)}{(1-t^2)^{d+1}}=\dfrac{\delta(P,t^2)+ta(t^2)}{(1-t^2)^{d+1}}.$$
%Therefore, we can denote $\delta(W,t)=h(P,t^2)+ta(t^2)$ and $h(P,t)$ is the even-numbered power and $ta(t^2)$ is the odd-numbered power of $\delta(W,t)$.
%Here, 
Since the degree of $E(W,t)$ is $-3$ as a rational function, 
the degree of $\delta(P,t^2)+ta(t^2)$ is $2d -1$.
Hence $\deg a(t) = d-1$ $(= \deg \delta (P, t))$.
Moreover, since the Ehrhart ring of $W$ is Gorenstein,
the coefficients of $\delta(P,t^2)+ta(t^2)$ are symmetric.
Thus $a(t) = t^{d-1}\delta(P,1/t)$ and 
$\delta(P,t^2)+ta(t^2) = \delta(P,t^2) + t^{2d-1} \delta(P,1/t^2)$. 
%By the Ehrhart reciprocity, it follows that
It is known that 
$E(P, 1/t) = -\sum_{k \ge 1}i(P,-k) t^{k}$ (see \cite[Chapter 4]{Rich}).
Hence
\begin{align*}
\sum_{k\ge 0} i^{\rm odd}(W,2k+1)t^{2k+1}
&=\dfrac{t^{2d-1}\delta(P,1/t^2)}{(1-t^2)^{d+1}}\\
&=\frac{(-1)^{d+1}}{t^3} \dfrac{\delta(P,1/t^2)}{ (1-1/t^2)^{d+1}}\\
&=\frac{(-1)^{d+1}}{t^3} E(P, 1/t^2)\\
&=\frac{(-1)^d}{t^3} \sum_{k\ge 1} i(P,-k) t^{2k}.
\end{align*}
Thus 
$
i^{\rm odd}(W,2k+1)
= (-1)^{d} i(P,-k-2)
= (-1)^{d} i^{\rm even}(W,-2k-4),
$
as desired.
\end{proof}

\begin{rei}
Let $W={\rm FRAC}(K_d)$ and $P={\mathcal P}(K_d)$
where $K_d$ is a complete graph with $d$ vertices.
It is known \cite[Example 27]{Steingrimsson} that
$\delta(P,t)=A_d(t)+dtA_{d-1}(t)$.
Let
$$
E(W,t) = \frac{b_0+b_1t+\dots+b_{2d-1}t^{2d-1}}{(1-t^2)^{d+1}}.
$$
%Then $b_0=1$ and 
%$b_i=A(d,\lfloor i/2 \rfloor)+dA(d-1, \lfloor (i-1)/2 \rfloor)$
% for $i = 1,\ldots, 2d-1$.
%In fact, since
Since
\begin{align*}
\delta(P,t)&=A_d(t)+dtA_{d-1}(t)\\
&=\sum_{i=0}^{d-1} A(d,i)t^i+dt\sum_{i=0}^{d-2}A(d-1,i)t^i\\
&=\sum_{i=0}^{d-1} A(d,i)t^i+d\sum_{i=1}^{d-1}A(d-1,i-1)t^i\\
&=1+\sum_{i=1}^{d-1} (A(d,i)+dA(d-1,i-1)) t^i
\end{align*}
hold, the $\delta$-vector $(\delta_0,\ldots,\delta_{d-1})$ 
of $P$ satisfies $\delta_0=1$ and $\delta_i=A(d,i)+dA(d-1,i-1)$
for $i =1,2,\ldots,d-1$.
By Theorem \ref{hpoly},
we can obtain the formula $b_0=1$ and 
$b_i=A(d,\lfloor i/2 \rfloor)+dA(d-1, \lfloor (i-1)/2 \rfloor) $
for $i = 1,2,\ldots,2d-1.$
\end{rei}

\begin{rei}
Let $W_d = {\rm FRAC}(C_d)$ where $C_d$ is an odd cycle of length $d$.
We computed the numerator $g(W_d,t)$
of $E(W_d, t)=g(W_d,t)/(1-t^2)^{d+1}$ for $d = 3,5,7, 9$ by using software {\tt Normaliz} (\cite{normaliz}).
\begin{align*}
g(W_3,t) &=1+4t+7t^2+7t^3+4t^4+t^5.
\end{align*}
\begin{align*}
g(W_5,t) &=1+11t+51t^2+131t^3+206t^4\\
&\ \ +206t^5+131t^6+51t^7+11t^8+t^9.\\
g(W_7,t)  & =1+29t+281t^2+1408t^3+4320t^4+8814t^5+12475t^6\\
 & \ \ +12475t^7+8814t^8+4320t^9+1408t^{10}+281t^{11}+29t^{12}+t^{13}.\\
g(W_9,t) & =1+76t+1450t^2+12844t^3+67000t^4+230986t^5+561004t^6\\
 & \ \ +996310t^7+1321369t^8+1321369t^9+996310t^{10}+561004t^{11}\\
 & \ \ +230986t^{12}+67000t^{13}+12844t^{14}+1450t^{15}+76t^{16}+t^{17}.
\end{align*}
\end{rei}

\section{The dual polytope of ${\mathcal Q}(G)$}

In this section, we will discuss the dual polytope ${\mathcal Q}(G)^{\vee}$ of ${\mathcal Q}(G)$.
%Here, ${\mathcal Q}(G)^{\vee}$ possesses the faces $P_G$ where $P_G$ is a edge polytope of a graph $G$. 
Recall that 
$$
{\mathcal Q}(G)^{\vee} = {\rm Conv} (\{{\bf e}_i+ {\bf e}_j \mid (i,j) \in E(G)\} \cup \{-{\bf e}_i \mid 1\le i \le d\})
$$
if $G$ has no isolated vertices.
It is easy to see that ${\mathcal Q}(G)^{\vee}$ is Fano. 
A lattice polytope $P \subset \RR^d$ is called {\em normal} if 
$\mathbb{Z}_{\ge 0} A = \mathbb{Q}_{\ge 0} A \cap  \mathbb{Z}A$,
where 
$$
A=
\begin{pmatrix} 
 {\bf a}_1 & \cdots & {\bf a}_n\\
 1 &\cdots& 1 
\end{pmatrix}
$$
such that $\{ {\bf a}_1, \ldots, {\bf a}_n\}
=
P \cap \ZZ^d $.
%\{ (\alpha, 1) \in \ZZ^{d+1} \ | \ \alpha \in P \cap \ZZ^d \}$.
Here $\mathbb{Z}A = \{\sum_{i=1}^n z_i ({\bf a}_i ,1) \ | \ z_i \in \mathbb{Z}\}$,
for example.
A triangulation $\Delta$ of $P$ is called {\em unimodular}
if the normalized volume of each maximal simplex of $\Delta$ is one.
If $\mathbb{Z}A = \mathbb{Z}^{d+1}$, then the normalized volume
of each maximal simplex is equal to the absolute value of the corresponding maximal minor of $A$.
See, \cite[Section 5.5]{dojo}.
It is known that a lattice polytope $P$ is normal if $P$ has a unimodular triangulation
(\cite[Theorem~5.6.7]{dojo}).

\begin{thm}
Let $G$ be a finite simple graph without isolated vertices. 
Then the following conditions are equivalent.
\begin{enumerate}
\item [{\rm (i)}] The graph $G$ is a bipartite graph{\rm ;}
\item [{\rm (ii)}] The dual polytope ${\mathcal Q}(G)^{\vee}$ has a unimodular
triangulation{\rm ;}
\item [{\rm (iii)}] The dual polytope ${\mathcal Q}(G)^{\vee}$ is normal{\rm ;}
\item [{\rm (iv)}] The dual polytope ${\mathcal Q}(G)^{\vee}$ is
a Gorenstein Fano polytope.
\end{enumerate}
\end{thm}

\begin{proof}
Since ${\mathcal Q}(G)^{\vee}$ is Fano, 
and since $({\mathcal Q}(G)^\vee)^\vee = {\mathcal Q}(G)$,
${\mathcal Q}(G)^{\vee}$ is Gorenstein Fano if and only if
${\mathcal Q}(G)$ is a lattice polytope.
By Proposition \ref{whenislattice},  ${\mathcal Q}(G)$ is a lattice polytope
if and only if $G$ is bipartite.
Hence we have {\rm (i)} $\Leftrightarrow$ {\rm (iv)}.
Moreover, (ii) $\Rightarrow$ (iii) holds in general.
Let $\mathcal{A}_G$ be the vertex-edge incidence matrix of $G$ and let
$\mathcal{A}'_G$ be the configuration matrix of ${\mathcal Q}(G)^{\vee}$, namely,
$$\mathcal{A}'_G=\begin{pmatrix} {\bf 0} & \mathcal{A}_G & -E_d \\1 & 1 \cdots 1 & 1 \cdots 1 \end{pmatrix},$$
where $E_d$ is an identity matrix.
Then $\mathbb{Z} \mathcal{A}'_G =\mathbb{Z}^{d+1}$.
Hence
${\mathcal Q}(G)^{\vee}$ is normal if and only if
$\mathbb{Z}_{\ge 0} \mathcal{A}'_G =
 \mathbb{Q}_{\ge 0} \mathcal{A}'_G \cap  \mathbb{Z}^{d+1}$.

{\rm (i) $\Rightarrow$ (ii):} 
Suppose that $G$ is bipartite.
It is known \cite{IP} that the vertex-edge incidence matrix of any bipartite graph is 
{\em totally unimodular}, i.e., the determinant of 
every square non-singular submatrix is $\pm 1$.
Hence it follows that the submatrix
$B=\begin{pmatrix} \mathcal{A}_G & -E_d  \end{pmatrix}$
of $\mathcal{A}'_G$ is totally unimodular.
Let $\Delta$ be a {\em pulling triangulation} (\cite{A}, \cite[Proposition~5.6.5]{dojo})
 of ${\mathcal Q}(G)^{\vee}$ such that the origin is a vertex of
every maximal simplex in $\Delta$.
Such a triangulation is obtained by a Gr\"obner basis of the toric ideal of $\mathcal{A}'_G$
with respect to a reverse lexicographic order such that the smallest variable corresponds
to the origin.
Then the normalized volume of each maximal simplex in $\Delta$ is equal to 
the absolute value of the corresponding maximal minor of $B$.
Since $B$ is totally unimodular, each maximal minor of $B$ is $\pm 1$, and hence
the triangulation $\Delta$ is unimodular.

(iii) $\Rightarrow$ (i): Suppose that the graph $G$ contains an odd cycle $C$.
% having no chords.
Now, we will show that ${\mathcal Q}(G)^{\vee}$ is not normal, that is, 
$\mathbb{Z}_{\ge 0} \mathcal{A}'_G \ne \mathbb{Q}_{\ge 0} \mathcal{A}'_G
\cap  \mathbb{Z}^{d+1}$.
%(It is easy to see that $\mathbb{Z}\mathcal{A}'_G = \mathbb{Z}^{d+1}$).
We may assume that $C=(1,2,\ldots, 2k+1)$.
Let 
\begin{eqnarray*}
{\bf u} &= &
\frac{1}{2}
\left(
{\bf e}_{d+1}+
({\bf e}_1 + {\bf e}_{2k+1}+ {\bf e}_{d+1})+
\sum_{i=1}^{2k} ({\bf e}_i + {\bf e}_{i+1}+ {\bf e}_{d+1})
\right)\\
& = &
(k+1) {\bf e}_{d+1} +
\sum_{i=1}^{2k+1} {\bf e}_i 
.
\end{eqnarray*}
%$$\begin{pmatrix} 1 \\ \vdots \\ 1 \\ 0 \\ \vdots \\ 0 \\ k+1 \end{pmatrix}:={\bf u}$$
Then ${\bf u}$ belongs to $\mathbb{Q}_{\ge 0} \mathcal{A}'_G \cap\mathbb{Z}^{d+1}$. 
It is enough to show that ${\bf u}\notin\mathbb{Z}_{\ge 0}\mathcal{A}'_G$.
Suppose 
\begin{equation}
\label{u}
{\bf u}= \gamma {\bf e}_{d+1}
+ \sum_{(i,j)\in E(G)} \alpha_{ij}({\bf e}_i + {\bf e}_{j}+ {\bf e}_{d+1})
+\sum_{i=1}^d \beta_i(-{\bf e}_i +  {\bf e}_{d+1}) %\tag{$\natural$}
\end{equation}
for some $\alpha_{ij},\beta_i \in \mathbb{Z}_{\ge 0}$.
Then the coefficient of ${\bf e}_i$ ($1 \le i \le 2k+1$) in (\ref{u}) is
$1=\sum_{(i,j)\in E(G)} \alpha_{ij}-\beta_i,$
and that of ${\bf e}_i$ ($2k+2 \le i \le d$) in (\ref{u}) is
$0=\sum_{(i,j)\in E(G)} \alpha_{ij}-\beta_i.$
By summing up the equations for $1\le i \le d$, we obtain
\begin{equation}
\label{ni}
2k+1=2\sum_{(i,j) \in E(G)}\alpha_{ij}-\sum_{i=1}^d \beta_i.
\end{equation}
On the other hand, the coefficient of ${\bf e}_{d+1}$ in (\ref{u}) is
\begin{equation}\label{san}
k+1=\gamma + \sum_{(i,j) \in E(G)}\alpha_{ij}+\sum_{i=1}^d \beta_i.
\end{equation}
Since $\gamma$ and $\sum_{i=1}^d \beta_i$ are nonnegative,
by equations (\ref{ni}) and (\ref{san}), we obtain
$$k+\dfrac{1}{2} \le \sum_{(i,j)\in E(G)}\alpha_{ij} \le k+1.$$
Since $\sum_{(i,j)\in E(G)} \alpha_{ij} \in \mathbb{Z},$
it follows that $\sum_{(i,j)\in E(G)} \alpha_{ij}=k+1$.
Hence, by equation (\ref{ni}), we have
$\sum_{i=1}^d \beta_i = 1$.
Thus, by equation (\ref{san}), we have $\gamma + 1 = 0$, which is a contradiction.
\end{proof}


\begin{thebibliography}{99}
\bibitem{A}
C.~A.~Athanasiadis, $h^*$-vectors, Eulerian polynomials and stable polytopes of graphs, 
{\it Electron. J. Combin.} {\bf 11} (2004/06), no. 2, Research Paper 6, 13 pp. (electronic). 

\bibitem{normaliz}
W. Bruns, B. Ichim, T. R\"omer, R. Sieg and C. S\"oger,
{\tt Normaliz}. 
Algorithms for rational cones and affine monoids. \\
Available at {\tt https://www.normaliz.uni-osnabrueck.de}.

\bibitem{BR}
W. Bruns and T. R\"omer,
$h$-Vectors of Gorenstein polytopes,
{\it J. Combin.~Theory~Ser.~A} {\bf 114} (2007), 65--76.


%\bibitem{Hibi}
%T. Hibi, ``Algebraic Combinatorics on Convex Polytopes,''
%Carslaw Publications, Glebe, N.S.W., Australia, 1992.

\bibitem{dojo}
T. Hibi, Ed., 
``Gr\"obner Bases: Statistics and Software Systems,''
Springer, 2013.

\bibitem{Hibi-Em} E. De Negri and T. Hibi, Gorenstein algebras of Veronese type, 
{\it J. Algebra} {\bf 193} (1997), 629--639.

\bibitem{NT}
G.~L.~Nemhauser and L.~E.~Trotter, Jr.,
Properties of vertex packing and independence system polyhedra, {\it Math. Programming} {\bf 6} (1974), 48--61. 

\bibitem{OH}
H. Ohsugi and T. Hibi,
Special simplices and Gorenstein toric rings,
{\it J. Combin. Theory Ser. A} {\bf 113} (2006), 718--725. 


\bibitem{Rich2}
R.~P.~Stanley, 
``Enumerative Combinatorics'' Volume 1 second edition,\\
Wadsworth \& Brook, Monterey, Wadsworth \& Brooks/Cole Math Series, 1986.



\bibitem{Rich}
R. P. Stanley, Two poset polytopes, 
{\it Discrete Comput. Geom.} {\bf 1} (1986), 9--23.

\bibitem{SV}
 J. Schepers and L. Van Langenhoven, 
Unimodality questions for integrally closed lattice polytopes, 
{\it Ann.~Comb.} {\bf 17} (2013), 571--589. 

\bibitem{IP} 
A. Schrijver, ``Theory of Linear and Integer Programming,''
John Wiley \& Sons, Ltd., Chichester, 1986.

\bibitem{Sta}
A. Stapledon, 
Inequalities and Ehrhart $\delta$-vectors,
{\it Trans. Amer. Math. Soc.} {\bf 361} (2009), 5615--5626.


\bibitem{Steingrimsson} 
E. Steingr\'{i}msson, A decomposition of 2-weak vertex-packing polytopes, 
{\it Discrete Comput. Geom.} {\bf 12} (1994), 465--479.






\end{thebibliography}
\end{document}